\documentclass[12pt]{article}
\usepackage[usenames]{color} 
\usepackage{amssymb} 
\usepackage{amsmath} 
\usepackage{amssymb,bm,amsmath,amssymb,amsthm}
\usepackage[]{inputenc}

\usepackage{cite}

\usepackage{hyperref}

\usepackage{authblk}

\newtheorem{theorem}{Theorem}

\begin{document}

\title{Curves with many points over finite fields: the class field theory approach}
\author{Pavel Solomatin \\
\texttt{p.solomatin@math.leidenuniv.nl}
}
\affil{Leiden University, Mathematical Department,\\
Niels Bohrweg 1, 2333 CA Leiden}

\maketitle

\begin{abstract}
The problem of constructing curves with many points over finite fields has received considerable attention in the recent years. Using the class field theory approach, we construct new examples of curves ameliorating some of the known bounds. More precisely, we improve the lower bounds on the maximal number of points $N_q(g)$ for many values of the genus $g$  and of the cardinality $q$ of the finite field $\mathbb F_q$, by looking at all unramified coverings of all genus three smooth projective curves over $\mathbb F_q$, for $q$ is an odd prime less than 19.
\end{abstract}

\textbf{\\ \\ \\ Acknowledgements:} 
First, I would like to thank both my advisors, namely professor Alexey Zykin and professor Alexey Zaytsev for their advices and many helpful discussions about the project. I also want to thank professors Sergey Galkin and Peter Bruin for providing access to the Magma computer algebra system. We use the database \cite{Quart1} of plane quartics over finite fields. The access to the database was given by Christophe Ritzenthaler. His advices and suggestions were also helpful. Finally, I would like to thank an anonymous reviewer of an earlier version of the paper for a lot of important comments.

\newpage

\section{Introduction to the problem}

In this paper $C$ is a smooth projective curve over a finite field. An interesting question is how many points there can be on a curve of given genus over a given finite field. Let $\#C(\mathbb F_{q})$ denote the number of points on $C$ over $\mathbb F_{q}$. We have the following classical result:
 
 \begin{theorem}[Weil bound] 
 \label{Weil} Let $g$ be the genus of the curve $C$, then
 $$|q+1 - \#C(\mathbb F_{q})| \le 2g\sqrt{q}.$$
\end{theorem}
 
 The problem of improving this bound is very difficult. Many well-known mathematicians, such as J.-P. Serre, V. Drinfeld and others, devoted considerable efforts to its study. A reason why this problem attracts much attention is that it has quite a few applications to the coding theory, cryptography, etc. 

 Let us denote by $N_{q}(g)$ the maximum of the number of points on a smooth projective curve $C$ of genus $g$ over $\mathbb F_{q}$. By using the Weil bound, one immediately sees that $N_{q}(g) \le q+1 + 2g\sqrt{q}$. 
 
 The goal of this paper is to find new lower bounds for the number $N_{q}(g)$ for small  $g$ and $q$, that is for $g \le 50$ and $q \le 19$.  Nowadays we do not know exactly the number $N_{q}(g)$ for many pairs $(q, g)$ from the above intervals, except for the cases $g=1$, 2, 3 or 4.  
 For example, $N_2(1)=5$, $N_2(2)=6$, $N_2(4)=8$, but $N_2(16)$ is either 17 or 18. Just a little improvement of the known bounds requires in many cases both modern mathematical tools and computer support.  
The  up-to-date tables for $N_{q}(g)$ are available on the web-site \href{http://www.manypoints.org/}{manypoints.org}.

 \section{Methods}
 \subsection{Serre's Example}
 
 Here we describe one of the ideas for constructing curves with many points. This idea belongs to Serre and its generalization is very important for our purposes. The following examples are taken from \cite{Hunt}.  
 
 \begin{theorem}[Serre]
 $N_2(4)\ge 8$ and $N_2(11) \ge 14$.
 \end{theorem}
 
 \begin{proof}
   Consider the elliptic curve $C$ defined over $\mathbb F_2$ by the equation $y^2 +y = x^3 + x$. It is easy to see that $C$ has exactly five rational points: $P_{0}=\infty$, $P_{1}=(0,0)$, $P_2=(1,0)$, $P_3=(1,1)$ and $P_4=(0,1)$ and the map $\phi$: $P_{i} \mapsto i$, from $C(\mathbb F_{2})$ to $ \mathbb Z/5 \mathbb Z$ is an isomorphism of abelian groups. Hence, there is a function $f$ on $C$  with the divisor $D=[a_0; a_1; a_2; a_3; a_4] = \sum a_i P_i$ if and only if $\sum a_i = 0 $ and $\sum i* a_i = 0 \mod 5$.
 
 So, the divisor $D_1=[-3,-1,2,1,1]$ is the divisor of some function $f_1 (x)$. One may take the Artin-Schreier  extension $C_1$: $w^2 + w = f_1(x)$. According to the Riemann-Hurwitz formula, it has genus 4. Moreover it is easy to see that it  has exactly eight rational points. This shows that $N_2(4) \ge 8$. For the same reason, the divisor $D_2=[-1,-3,1,1,2]$ is the divisor of some function $f_2(x)$. The Artin-Schreier covering $C_2$: $w^2+w=f_2(x)$ also has genus 4 and eight rational points. On can take the fibre product of $C_1$ and $C_2$. It is  a genus 11 curve with 14 rational points.  
 \end{proof}

So, we have $N_2(4)\ge 8$ and $N_2(11) \ge 14$. Let us compare this result with the Weil-bound: $$N_2	(4) \le 2+1+8\sqrt{2} \approx 14.313$$  and $$N_2(11) \le 2 + 1 + 22\sqrt{2} \approx 34.11.$$ 

It is possible to improve this bound. The point is that in both cases mentioned above $g$ is ``sufficiently greater''  than $q$. Let us illustrate this idea by the following theorem: 

\begin{theorem}[Ihara bound]
We have $$ N_{q}(g) \le q+1 + 1/2 (\sqrt{(8q+1)g^2 + 4g(q^2-q)} - g ).$$ 
\end{theorem}
In particular if $g \ge \frac{\sqrt{q}(\sqrt{q}-1)}{2}$, then Ihara bound improves the Weil bound.  

\begin{proof}
See \cite{Tsf} or \cite{Vir1}.
\end{proof}

Let us take here $q=2$ and $g=4$, $11$. The Ihara bound gives us: 

$$ N_{2}(4) \le 1 + 2\sqrt{19} \approx 9.7177$$ 
and

$$ N_2(11) \le \frac{1}{2}(\sqrt{2145} -5) \approx 20.6571.$$

Actually $N_2(4)=8$ and $N_2(11)=14$, but it is quite a non-trivial result, which requires the so called Weil-explicit formulae and Oesterl\'e bound. See \cite{Voight1}.

\subsection{A generalization}
Here we describe a generalization of the above method. This generalization has been used in many papers (e.g. \cite{Cur}). For clarity we save the original notation. 

First of all, we notice that the equation $w^2+w=f_1(x)$ from the above theorem gives an abelian extension where the points $(x,y)$ such that $f_1(x)=0$ split completely. In contrast to the above example, for the sake of clarity, we restrict ourselves to working with unramified abelian covers. The unramified abelian covers of $C$ are parameterized by the subgroups of the Picard group $Pic^{0}(C)$. Thus, we find ourselves working with the class field theory.

Let us fix the ground field $k= \mathbb F_{q}$ and let us consider a curve $C$ of genus $g$ over $k$. We denote by $F$ the function field $\mathbb F_q(C)$. By using class field theory we want to construct abelian extensions $F'$ of $F$ corresponding to new curves $C'$ with many points. 

Let  $S = \{P_{i} \}$ be the set of all rational points of $C$. Let us fix one rational point $O$ and consider the $S_{O}$-Hilbert class field which we denote by $F_{O}$. By definition, $F_O$ is the maximal unramified abelian extension of $F$ in which $O$ splits completely.

The class field theory give us the isomorphism:

$$\phi \colon Pic^{0}(C) \to Gal(F_O/F).$$

The isomorphism $\phi$ maps the class $[P- \deg(P)O]$ to the Artin symbol of $P$(the Frobenius map at point $P$).  
By using this isomorphism, one produces new curves. Namely, if $G$ is a subgroup of $Pic^{0}(C)$ of index $d=[Pic^{0}(C):G]$, then its image $\phi(G)$ is a subgroup of $Gal(F_O / F)$.  Hence, we have the subfield $F_{O}^{\phi(G)}$ of $F_O$ fixed by $\phi (G)$. By definition of $\phi$ a rational point $P$ splits completely if and only if $[P-\deg(P) O] \in G$. This gives us an unramified abelian extension, that corresponds  to a new curve. Since  the extension is unramified we can control its genus and its number of rational points.  

Thus, for any pair $[G,O]$ we get the unramified extension $F_{O}^{\phi(G)}$ of degree $d$. It has genus $d(g-1)+1$ and $d \cdot | G \cap \{ [P_i - O] \} |$ rational points. 

By searching through all the subgroups of $Pic^{0}(C)$ and through  all the points of $C$, one can improve the known lower bounds on $N_q(g)$. Note that for a fixed $G$ and different places $O$ we get different extensions. Despite the fact that such extensions have isomorphic Galois groups, the corresponding curves may have different numbers of rational points. We will see concrete examples in the next section.

\section{Examples and Calculations}

\subsection{ How to organize the computation}

Previously, the above method was used mostly for hyper-elliptic curves of small genus.  For example, for the case of genus two curves over finite fields $\mathbb F_q$ with cardinality $q \le 16$ this search was implemented in \cite{Cur}. In our search, we do the same calculation for all genus three curves. We use Magma software to perform a computer search.

If we have a genus $3$ curve $C$ then it is either a hyper-elliptic curve, or a  plane quartic. 

 A plane quartic is a geometrically smooth projective curve given by an equation of the form $f(x,y,z)=0$, where  $f(x,y,z)\in \mathbb F_{q}[x,y,z]$ is a homogeneous polynomial degree 4. By using Invariant theory, it is possible to describe very explicitly  the ``moduli space'' of plane quartics over finite fields. This was done by Lercier, Ritzenthaler, Rovetta, and Sijsling in \cite{Quart1}.

Otherwise, if $C$ is a hyper-elliptic curve, then its affine part has a smooth model $y^2 = f(x)$, where $f(x)$ is a polynomial of degree 7 or 8 without multiple roots over algebraic closure of the ground field. To obtain a projective model one has to take the normalization of the projective closure $C$.  A good thing here is that Magma allows one to work with the projective model $C$ when only $f$ is given. But unfortunately, we do not have an analogue of the database \cite{Quart1} in this case, so it takes much machine resources to provide calculation for all hyper-elliptic curves for a given genus.  
 
Finally, we run the algorithm for the base field $\mathbb F_p$, for the prime $p$ equal to 3, 5, 7, 11, 13, 17 and 19. We refer the reader to the next section for details.

Let us show how it works for some concrete examples.

\subsection{Concrete examples}

Unfortunately, for $p=3$  we could not find any new curves. 
Otherwise, we can take the base field to be $\mathbb F_{5}$, $\mathbb F_{7}$, $\mathbb F_{11}$, $\mathbb F_{13}$, $\mathbb F_{17}$ or $\mathbb F_{19}$ and provide many examples improving previously known bounds. 

Let us consider the case $ k=\mathbb F_{7}$. By ranging over plane quartics , we have found the following examples.

Let $C$ be the curve defined in $\mathbb P^2$ by the equation $ 6x^4 + y^3z + 6x^2z^2 + 4xz^3 + 6z^4=0$. The curve $C$ has exactly 14 rational points, namely: 
$P_0 =(0 : 1 : 1)$, $P_0' = (0 : 2 : 1)$, $P''_0 =  (0 : 4 : 1)$,
$P_1=(1 : 3 : 1)$, $P_1'= (1 : 5 : 1)$, $P_1''=(1 : 6 : 1)$,
$P_2=(2 : 3 : 1)$, $P_2'=(2 : 5 : 1)$, $P_2''=(2 : 6 : 1)$,
$P_5=(5 : 1 : 1)$, $P_5'=(5 : 2 : 1)$, $P_5''=(5 : 4 : 1)$,
$P_6=(6: 0 : 1)$,  $\infty = (0 : 1 : 0)$.

According to Magma,  $Pic^0(C)$ is an abelian group isomorphic to $\mathbb Z/{903\mathbb Z}$. Let us take the index 7 subgroup $G$ spanned by $7$, which is isomorphic to $\mathbb Z/{129\mathbb Z}$, and $O = P_0$. Then it gives us a curve of genus 15 with 56 rational points. This improves the previous lower bound for $g=15$, which was 52.

Next, let us consider the curve $C$ over $\mathbb F_7$ defined by the equation $6x^4 + y^3z + 2x^2z^2 + z^4=0$. It has 12 rational points: 
$P_0=(0 : 3 : 1)$, $P_0'=(0 : 5 : 1)$, $P_0''= (0 : 6 : 1)$,
$P_2=(2 : 0 : 1)$, 
$P_3=(3 : 3 : 1)$, $P_3'=(3 : 5 : 1)$, $P_3''=(3 : 6 : 1)$, 
$P_4=(4 : 3 : 1)$, $P_4'=(4 : 5 : 1)$, $P_4''=(4 : 6 : 1)$, 
$P_5=(5 : 0 : 1$), $\infty= (0 : 1 : 0)$.

Its class group is isomorphic to $\mathbb Z /3 \mathbb Z \oplus \mathbb Z /6 \mathbb Z \oplus \mathbb Z /42 \mathbb Z$. If one takes $O=P_2$ and the subgroup $G$ spanned by $v=(1,2,2)$, then we get a curve of genus 9 with 36 rational points. The previous bound for $N_7(9)$ was 32.

Let us change the base field and take $\mathbb F_{13}$. Here we have also found some interesting examples improving the previous bounds. For instance, consider the curve $C$ defined by $12x^4 + y^3z + z^4=0$ which has the class group $(\mathbb Z /3 \mathbb Z)^3 \oplus( \mathbb Z /12 \mathbb Z)^2$. One takes  $O = (0 : 4 : 1)$ and the group $G$ defined as follows: let $m_1 = g(2)$, $m_2 = g(4)$, $m_3 = g(1) + g(5)$, where $g(k)$ denotes the generator of the $k$-th component of $Pic^{0}(C)$ in the above decomposition. Let $G$ be spanned by $m_1$,$m_2$ and $m_3$. Then, we get an extension of genus 19 with 108 rational points. The previous bound for $N_{13}(19)$ was $90$.

We now summarize the results of our computations. 

\section{Main results}
In this part we summarize all results we have found during this research. We provide this data in the following way: we give the genus and the number of rational points of a certain covering of the base curve. We also give the previously known bound to compare it with value we obtained.  

\begin{table}[]
\centering
\caption{Results for the base fields $\mathbb F_5$,$\mathbb F_7$, $\mathbb F_{11}$ }
\label{my-label}	
\begin{tabular}{|l|l|l|r|}
\hline
The base Field   & Genus & Improvements & Previous Result \\ \hline
$\mathbb F_5$    & 25    & 60         & 55--66                 \\ \hline
$\mathbb F_7$    & 9     & 36         & 32--41                 \\ \hline
$\mathbb F_7$    & 15    & 56         & 60--77                 \\ \hline
$\mathbb F_7$    & 21    & 70         & 60--77                 \\ \hline
$\mathbb F_7$    & 29    & 84         & unknown--98            \\ \hline
$\mathbb F_7$    & 31    & 90         & unknown--103           \\ \hline
$\mathbb F_7$    & 33    & 96         & unknown--109           \\ \hline
$\mathbb F_7$    & 35    & 102        & unknown--114           \\ \hline
$\mathbb F_7$    & 37    & 108        & unknown--119           \\ \hline
$\mathbb F_7$    & 39    & 95         & unknown--125           \\ \hline
$\mathbb F_7$    & 43    & 105        & unknown--135           \\ \hline
$\mathbb F_7$    & 45    & 110        & unknown--140           \\ \hline
$\mathbb F_7$    & 47    & 115        & unknown--145           \\ \hline
$\mathbb F_7$    & 49    & 120        & 114--150               \\ \hline
$\mathbb F_{11}$ & 9     & 52         & 48--59                 \\ \hline
$\mathbb F_{11}$ & 13    & 66         & 60--77                 \\ \hline
$\mathbb F_{11}$ & 21    & 90         & 80--110                \\ \hline
$\mathbb F_{11}$ & 23    & 99         & 88--119                \\ \hline
$\mathbb F_{11}$ & 25    & 108        & 96--127                \\ \hline
$\mathbb F_{11}$ & 27    & 104        & unknown--135           \\ \hline
$\mathbb F_{11}$ & 31    & 120        & unknown--139           \\ \hline
$\mathbb F_{11}$ & 35    & 136        & unknown--164           \\ \hline
$\mathbb F_{11}$ & 37    & 144        & unknown--171           \\ \hline
$\mathbb F_{11}$ & 43    & 168        & unknown--192           \\ \hline
$\mathbb F_{11}$ & 45    & 176        & unknown--199           \\ \hline
$\mathbb F_{11}$ & 47    & 161        & unknown--206           \\ \hline
$\mathbb F_{11}$ & 49    & 192        & unknown--213           \\ \hline
\end{tabular}
\end{table}

\begin{table}[]
\centering
\caption{Results for the base fields $\mathbb F_{13}$, $\mathbb F_{17}$}
\label{my-label}
\begin{tabular}{|l|l|l|l|}
\hline
The base field   & Genus & Improvements & Previous Result \\ \hline
$\mathbb F_{13}$ & 5     & 42         & 40--44                 \\ \hline
$\mathbb F_{13}$ & 19    & 108        & 90--115                 \\ \hline
$\mathbb F_{13}$ & 23    & 110        & unknown--133           \\ \hline
$\mathbb F_{13}$ & 25    & 120          & unknown--142    \\ \hline
$\mathbb F_{13}$ & 27    & 130          & unknown--152    \\ \hline
$\mathbb F_{13}$ & 31    & 135          & unknown--170    \\ \hline
$\mathbb F_{13}$ & 33    & 160          & 128--179        \\ \hline
$\mathbb F_{13}$ & 37    & 162          & 144--195        \\ \hline
$\mathbb F_{13}$ & 39    & 152          & unknown--203    \\ \hline
$\mathbb F_{13}$ & 41    & 180          & 160--211        \\ \hline
$\mathbb F_{13}$ & 47    & 184          & unknown--235    \\ \hline
$\mathbb F_{13}$ & 49    & 192          & unknown--243    \\ \hline
$\mathbb F_{17}$ & 5     & 48           & unknown--53     \\ \hline
$\mathbb F_{17}$ & 7     & 60           & unknown--70     \\ \hline
$\mathbb F_{17}$ & 9     & 72           & unknown--83     \\ \hline
$\mathbb F_{17}$ & 11    & 80           & unknown--96     \\ \hline
$\mathbb F_{17}$ & 13    & 90           & unknown--107    \\ \hline
$\mathbb F_{17}$ & 15    & 98           & unknown--118    \\ \hline
$\mathbb F_{17}$ & 17    & 112          & unknown--129    \\ \hline
$\mathbb F_{17}$ & 19    & 117          & unknown--140    \\ \hline
$\mathbb F_{17}$ & 21    & 130          & unknown--150    \\ \hline
$\mathbb F_{17}$ & 23    & 132          & unknown--161    \\ \hline
$\mathbb F_{17}$ & 25    & 144          & unknown--172    \\ \hline
$\mathbb F_{17}$ & 27    & 143          & unknown--183    \\ \hline
$\mathbb F_{17}$ & 29    & 154          & unknown--194    \\ \hline
$\mathbb F_{17}$ & 31    & 165          & unknown--205    \\ \hline
$\mathbb F_{17}$ & 33    & 192          & unknown--216    \\ \hline
$\mathbb F_{17}$ & 35    & 187          & unknown--226    \\ \hline
$\mathbb F_{17}$ & 37    & 216          & unknown--237    \\ \hline
$\mathbb F_{17}$ & 39    & 190          & unknown--248    \\ \hline
$\mathbb F_{17}$ & 41    & 240          & unknown--258    \\ \hline
$\mathbb F_{17}$ & 43    & 210          & unknown--269    \\ \hline
$\mathbb F_{17}$ & 45    & 220          & unknown--280    \\ \hline
$\mathbb F_{17}$ & 47    & 207          & unknown--291    \\ \hline
$\mathbb F_{17}$ & 49    & 240          & unknown--301    \\ \hline

\end{tabular}
\end{table}

\begin{table}[]
\centering
\caption{Results for the base field $\mathbb F_{19}$}
\label{my-label}
\begin{tabular}{|l|l|l|l|}
\hline
The base field   & Genus & Improvements & Previous Bound \\ \hline
$\mathbb F_{19}$ & 5     & 54           & unknown--60     \\ \hline
$\mathbb F_{19}$ & 7     & 66           & unknown--76     \\ \hline
$\mathbb F_{19}$ & 9     & 80           & unknown--92    \\ \hline
$\mathbb F_{19}$ & 11    & 90           & unknown--104   \\ \hline
$\mathbb F_{19}$ & 13    & 96           & unknown--117   \\ \hline
$\mathbb F_{19}$ & 15    & 112          & unknown--128   \\ \hline
$\mathbb F_{19}$ & 17    & 112          & unknown--140   \\ \hline
$\mathbb F_{19}$ & 19    & 126          & unknown--153   \\ \hline
$\mathbb F_{19}$ & 21    & 140          & unknown--164   \\ \hline
$\mathbb F_{19}$ & 23    & 143          & unknown--175   \\ \hline
$\mathbb F_{19}$ & 25    & 168          & unknown--186   \\ \hline
$\mathbb F_{19}$ & 27    & 156          & unknown--198   \\ \hline
$\mathbb F_{19}$ & 29    & 196          & unknown--209   \\ \hline
$\mathbb F_{19}$ & 31    & 180          & unknown--221   \\ \hline
$\mathbb F_{19}$ & 33    & 192          & unknown--232   \\ \hline
$\mathbb F_{19}$ & 35    & 204          & unknown--244   \\ \hline
$\mathbb F_{19}$ & 39    & 209          & unknown--267   \\ \hline
$\mathbb F_{19}$ & 43    & 231          & unknown--289   \\ \hline
$\mathbb F_{19}$ & 45    & 242          & unknown--301   \\ \hline
$\mathbb F_{19}$ & 47    & 253          & unknown--312   \\ \hline
\end{tabular}
\end{table}

\newpage

\bibliography{mybib}{}

\begin{thebibliography}{1}

\bibitem{Vir1}
Virgile Ducet and Claus Fieker.
\newblock Computing equations of curves with many points.
\newblock In {\em A{NTS} {X}---{P}roceedings of the {T}enth {A}lgorithmic
  {N}umber {T}heory {S}ymposium}, volume~1 of {\em Open Book Ser.}, pages
  317--334. Math. Sci. Publ., Berkeley, CA, 2013.

\bibitem{Quart1}
Reynald Lercier, Christophe Ritzenthaler, Florent Rovetta, and Jeroen Sijsling.
\newblock Parametrizing the moduli space of curves and applications to smooth
  plane quartics over finite fields.
\newblock {\em LMS J. Comput. Math.}, 17(suppl. A):128--147, 2014.

\bibitem{Cur}
Karl R{\"o}kaeus.
\newblock Computer search for curves with many points among abelian covers of
  genus 2 curves.
\newblock In {\em Arithmetic, geometry, cryptography and coding theory}, volume
  574 of {\em Contemp. Math.}, pages 145--150. Amer. Math. Soc., Providence,
  RI, 2012.

\bibitem{Tsf}
Michael Tsfasman, Serge Vl{\u{a}}du{\c{t}}, and Dmitry Nogin.
\newblock {\em Algebraic geometric codes: basic notions}, volume 139 of {\em
  Mathematical Surveys and Monographs}.
\newblock American Mathematical Society, Providence, RI, 2007.

\bibitem{Hunt}
Gerard van~der Geer.
\newblock Hunting for curves with many points.
\newblock In {\em Coding and cryptology}, volume 5557 of {\em Lecture Notes in
  Comput. Sci.}, pages 82--96. Springer, Berlin, 2009.

\bibitem{Voight1}
John Voight.
\newblock Curves over finite fields with many points: an introduction.
\newblock In {\em Computational aspects of algebraic curves}, volume~13 of {\em
  Lecture Notes Ser. Comput.}, pages 124--144. World Sci. Publ., Hackensack,
  NJ, 2005.

\end{thebibliography}
\bibliographystyle{plain}

\newpage

\tableofcontents

\end{document}